\documentclass[reqno,12pt]{amsart}

\usepackage{microtype}
\usepackage{amsfonts, amsmath, amsthm, amstext, amssymb, amscd, color, rotating}
\usepackage{url}
\usepackage{paralist}
\usepackage{tikz}
\usepackage[english]{babel}
\usepackage{mathtools}
\usepackage{framed}
\usepackage{amsrefs}

\usepackage{hyperref} 
\usepackage{enumerate}

\newtheorem{lemma}{Lemma}[section]
\newtheorem{theorem}[lemma]{Theorem}

\newtheorem{corollary}[lemma]{Corollary}
\newtheorem{remark}[lemma]{Remark}

\begin{document}

\title[Fridman Function, Injectivity Radius Function and Squeezing Function]{Fridman Function, Injectivity Radius Function and Squeezing Function}

\author{Tuen Wai Ng }
\address{The University of Hong Kong, Pokfulam, Hong kong SAR, PRC}
\email{ntw@maths.hku.hk}

\author{Chiu Chak Tang}
\address{The University of Hong Kong, Pokfulam, Hong kong SAR, PRC}
\email{chiuchakTang@gmail.com}

\author{Jonathan Tsai}
\address{The University of Hong Kong, Pokfulam, Hong kong SAR, PRC}
\email{jonathan.tsai@cantab.net}

\date{\today}



\maketitle

\begin{abstract}
Very recently, the Fridman function of a complex manifold $X$ has been identified as a dual of the squeezing function of $X$. In this paper, we prove that the Fridman function for certain hyperbolic complex manifold $X$ is bounded above by the injectivity radius function of $X$. This result also suggests us to use the Fridman function to extend the definition of uniform thickness to higher-dimensional hyperbolic complex manifolds. We also establish an expression for the Fridman function (with respect to the Kobayashi metric) when $X = \mathbb{D} \diagup \Gamma$ and $\Gamma$ is a torsion-free discrete subgroup of isometries on the standard open unit disk $\mathbb{D}$. Hence, explicit formulae of the Fridman functions for the annulus $A_r$ and the punctured disk $\mathbb{D}^*$ are derived. These are the first explicit non-constant Fridman functions. Finally, we explore the boundary behaviour of the Fridman functions (with respect to the Kobayashi metric) and the squeezing functions for regular type hyperbolic Riemann surfaces and planar domains respectively.
\end{abstract}


\section{Introduction}
\label{sect:intro}

Let $\mathbb{B}^n(a;r)$ be an $n$-th dimensional Euclidean open ball in $\mathbb{C}^n$ with center $a$ and radius $r$. When $a=0$ and $r=1$, we denote $\mathbb{B}^n(0;1)$ by $\mathbb{B}^n$ and $\mathbb{B}^1$ by $\mathbb{D}$. 
Let $X$ be an $n$-dimensional complex manifold. For any $z_1,z_2 \in X$, let $c_X (z_1,z_2)$ be the Carath\'{e}odory pseudo-distance between $z_1$ and $z_2$ and $k_X (z_1,z_2)$ be the Kobayashi pseudo-distance between $z_1$ and $z_2$. For $d=k$ or $c$, a complex manifold $X$ is said to be \textit{$d$-hyperbolic} if the pseudo-distance $d_X$ is indeed a distance on $X$. For any $z \in X$ and any $r>0$, denote $B^k_X (z;r)$ the open Kobayashi ball in $X$ centred at $z$ with radius $r$.

Let $X$ be an $n$-dimensional $k$-hyperbolic complex manifold. In 1983, Fridman \cite{paper:fri_def_ori} introduced the \textit{Fridman invariant} $h_X (z, \mathbb{B}^n)$, which is defined as 
\begin{equation*}
    h_X (z,\mathbb{B}^n)= \inf \lbrace 1/q_{X,f} (z, \mathbb{B}^n) \: :  f \in \mathcal{U}(\mathbb{B}^n,X)\rbrace
\end{equation*}
where 
\begin{equation*}
    q_{X,f}(z,\mathbb{B}^n)=\sup \lbrace r \, : \, B^k_X (z;r) \subset f(\mathbb{B}^n) \rbrace
\end{equation*}
and $\mathcal{U}(\mathbb{B}^n, X)$ denotes the family of all injective holomorphic functions from $\mathbb{B}^n$ to $X$. Notice that in 1979, Fridman \cite{Fri79} also introduced similar biholomorphic invariant when $\mathbb{B}^n$ is replaced by the unit polydisk and $B^k_X (z;r)$ is replaced by the corresponding Carath\'{e}odory ball.
The Fridman invariant is interesting because it gives some geometric information about the manifold. For instance, Fridman \cite{paper:fri_def_ori} showed that if a connected $k$-hyperbolic complex manifold $X$ has the property that $h_X (z,\mathbb{B}^n)=0$ for some $z \in X$, then $h_X (z,\mathbb{B}^n)=0$ for all $z \in X$ and $X$ is biholomorphic to $\mathbb{B}^n$. 
Also in \cite{paper:fri_def_ori}, he showed that if a bounded strictly pseudoconvex domain $X$ has $\mathcal{C}^3$ boundary, then $\lim\limits_{z \to \partial X} h_X (z,\mathbb{B}^n) =0$. See \cites{Fri79,paper:fri_def_ori} for more applications and properties of $h_X (z,\mathbb{B}^n)$ as well as its Carath\'{e}odory analog.

In 2019, Mahajan and Verma \cite{mahajan2019comparison} identified the Fridman invariant $h_X (z,\mathbb{B}^n)$ as a dual to the squeezing function $S_X (z)$, which can be reformulated as
\begin{equation*}
    S_X (z) 
= 
\sup
\left\lbrace 
\tanh \frac{r}{2}
\: : \:
B^k_{\mathbb{B}^n} (f(z);r) \subset  f(X )
, f \in \mathcal{U} (X, \mathbb{B}^n )
\right\rbrace
\end{equation*}
(see the Appendix for the more common definition of $ S_X (z) $ first introduced in \cite{squ_def1} and how to obtain the above reformulation). 
To see the duality between the squeezing function and the Fridman invariant, Nikolov and Verma
\cite{verma2019}, and independently, Deng and Zhang \cite{deng2018fridman} considered a modification $H_{X}^k (z)$ of $h_X (z,\mathbb{B}^n)$, which is defined to be
\begin{equation*}
    H_{X}^k (z)
:=
\sup
\left\lbrace
\tanh \frac{r}{2} \: : \:
B^k_X (z;r) \subset f(\mathbb{B}^n), f \in \mathcal{U}(\mathbb{B}^n,X)
\right\rbrace.
\end{equation*}
We will call $H_X^k (z)$ the \textit{Fridman function} of $X$ (with respect to the Kobayashi metric). Similarly, its Carath\'{e}odory analog $H^c_X (z)$ can be defined as
\begin{equation*}
    H_{X}^c (z)
=
\sup
\left\lbrace
\tanh \frac{r}{2} \: : \:
B^c_X (z;r) \subset f(\mathbb{B}^n), f \in \mathcal{U}(\mathbb{B}^n,X)
\right\rbrace.
\end{equation*}
Here, $B^c_\Omega (z;r)$ denotes the open Carath\'{e}odory ball in $X$ centred at $z$ with radius $r$. In \cite{verma2019}, Nikolov and Verma showed that
\begin{equation}
\label{eq:fri_order}
S_X (z) \leq H^c_X (z) \leq H^k_X (z)
\end{equation}
for any domain $X \subset \mathbb{C}^n$.

Let $X$ be a $d$-hyperbolic complex manifold. Let $\iota_X^k (z)$ be the \textit{injectivity radius function} at a point $z \in X$ with respect to the Kobayashi metric, which is defined to be, 
\begin{equation*}
    \iota_X^k (z)
    =
    \sup \left\lbrace \tanh \frac{r}{2}  \: : \: 
    \mbox{$B^k_X (z;r)$ is simply connected} 
        \right\rbrace.
\end{equation*}
Similarly, its Carath\'{e}odory analog $\iota_X^c (z)$ is given by
\begin{equation*}
    \iota_X^c (z)
    =
    \sup \left\lbrace \tanh \frac{r}{2}  \: : \: 
    \mbox{$B^c_X (z;r)$ is simply connected} 
        \right\rbrace.
\end{equation*}
For $d=k$ or $c$, the \textit{injectivity radius} $\iota_X^d$ of $X$ with respect to $d_X$ is defined by 
\[
\iota_X^d = \inf_{z \in X} \iota_X^d (z).
\]
For more information about injectivity radius, see for example \cites{inject1,book:hyper_geo,inject2}.

The following theorem relates the Fridman function $H^d_X (z)$ of $X$ and the injectivity radius function $\iota_X^d(z)$. 

\begin{theorem}
\label{thm:fri<iota}
Let $X$ be an $n$-dimensional $d$-hyperbolic complex manifold. Then the following three statements are true;

\begin{enumerate}
\item if $n=1$, $d=k$ or $c$, then $H^d_X (z) \geq \iota^d_X (z) $ for all $z\in X$;

\item if $n \geq1$ and $X = D \diagup \Gamma$, where $D \subset \mathbb{C}^n$ is a $k$-hyperbolic domain with the property that all open Kobayashi balls $B^k_D (z;r)$ of $D$ are simply connected and $\Gamma$ is a torsion-free discrete group of isometries of $(D,k_D)$, then $H^k_X (z) \leq \iota^k_X (z) $ for all $z\in X$;

\item when $n=1$, $H^k_X (z) = \iota^k_X (z) $ for all $z\in X$.

\end{enumerate}
\end{theorem}

\noindent
\begin{remark} \label{ball}
For $d=k$ or $c$, it is known that if $D \subset \mathbb{C}^n$ is a convex domain, then $k_D=c_D$ and any $d$-ball $B^d_D(z;r)$ of $D$ is convex and hence simply connected (see Corollary 4.8.3 and Theorem 4.8.13 of \cite{book:hyper_Kobayashi} for bounded $D$ and Lemma 3.1 and Proposition 3.2 of \cite{Bracci2009} for the unbounded case). Also notice that by the Hermann Convexity Theorem (page 286 of \cite{wolf1972fine}), any bounded symmetric domain is convex.
\end{remark}

When $n=1$, a Riemann surface $X$ is said to be \textit{uniformly thick} if its injectivity radius function $\iota^k_X (z)$ has a positive lower bound. For example, all bounded simply connected domains in $\mathbb{C}$ are uniformly thick whereas punctured domains in $\mathbb{C}$ are not. For more information about uniform thickness, see for example \cites{thick,book:hyper_geo}. We now extend the definition of uniform thickness to higher dimension $d$-hyperbolic complex manifolds as follows. For $d=k$ or $c$, we define a $d$-hyperbolic complex manifold of dimension $n$ to be $d$-\textit{uniformly thick} if its Fridman function $H^d_X(z)$ has a positive lower bound. Note that when $n=1$ and $d=k$, this definition coincides with the conventional one by part 3 of Theorem \ref{thm:fri<iota}. On the other hand, in \cite{squ_def1}, Deng, Guan and Zhang defined a bounded domain $X \subset \mathbb{C}^n$ to be \textit{holomorphic homogeneous regular} \cite{liu2004canonical} or with \textit{uniform squeezing property} \cite{yeung2009} if its squeezing function $S_X(z)$ has a positive lower bound.

Because of inequality (\ref{eq:fri_order}), we also have the following corollary.

\begin{corollary}
If a domain $X \subset \mathbb{C}^n$ has the uniformly squeezing property (i.e., $S_X (z) >c >0$ for some constant $c$  and all $z \in X$), then it is both $c$-uniformly thick and $k$-uniformly thick. (See Theorem 2 of Yeung \cite{yeung2009} for a more general result when $X$ is not a subset of $\mathbb{C}^n$.)
\end{corollary}

In our recent paper \cite{ourpaper_squeezing_annulus}, we showed that 
\begin{equation*}
    S_{A_r} (z) 
= 
\max
\left\lbrace 
|z|,
\frac{|z|}{r}
\right\rbrace
\end{equation*}
where $A_r= \{ z \: : \: 0<r<|z|<1  \}$ and this gives the precise form of $S_X (z)$ for all bounded non-degenerate doubly-connected domain $X \subset \mathbb{C}$ up to biholomorphism. 
Different proofs of this result based on the methods of harmonic measures and quadratic differentials are given by Gumenyuk and Roth \cite{Roth} and Solynin \cite{Solynin} respectively. Note that when $r\to 0$, we have $S_{\mathbb{D}^*} (z) =|z|$ where $\mathbb{D}^*= \mathbb{D} \setminus \{ 0 \} $. For any bounded homogeneous domain in $\mathbb{C}^n$, both its Fridman function and squeezing function are constant (see \cite{paper:fri_def_ori} and \cite{squ_def1} for some examples of these constant functions). So far no non-constant Fridman function has been explicitly constructed.  In this paper, we will construct for the first time several explicit non-constant Fridman functions. Indeed we will obtain the explicit expressions of $H_{A_r}^d$ and $H_{\mathbb{D}^*}^d$ by applying Theorem \ref{Thm:fri_upper} and \ref{thm:fri:main} below.

\begin{theorem}
\label{Thm:fri_upper}
For $d=k$ or $c$, let $D \subset \mathbb{C}^n$ be a convex domain which contains no complex affine lines and $\Gamma$ be a torsion-free discrete group of isometries of $(D,d_D)$. Let $X = D \diagup \Gamma$ and $\pi$ be the quotient map. For any $z\in X$, let $w\in D$ be any point such that $\pi(w)=z$. Then we have

\begin{equation}
\label{eq:fri_tanh}
H^d_X (z) \leq \min\limits_{\gamma \in \Gamma \setminus \lbrace \mathrm{Id} \rbrace } \tanh \frac{ d_{D} (w,\gamma (w) )}{4}
\end{equation}
\end{theorem}

\noindent
\begin{remark} \label{k-hopf}
 If we assume that $D \subset \mathbb{C}^n$ is a complete $k$-hyperbolic domain, then the inequality in Theorem \ref{Thm:fri_upper} still holds for $d=k$. This is because the topology induced by the Kobayashi distance $k_D$ is the same as the Euclidean topology of $D$ (cf. Theorem 3.2.1 of \cite{book:hyper_Kobayashi}). Hence, $(D,k_D)$ is a complete locally compact metric space. One can then follow the proof of Theorem \ref{Thm:fri_upper} to obtain the inequality.
\end{remark}

\medskip
Suppose that $n=1$ and $d=k$. 
Let $\mathbb{D}=\left\lbrace w  \in \mathbb{C}  \: : \:  |w| <1 \right\rbrace$ be the standard open unit disk in $\mathbb{C}$ and let $\mathbb{H}=\left\lbrace w  \in \mathbb{C}  \: : \:  \mathrm{Im}(w) >0 \right\rbrace $
be the upper half plane. Note that both $\mathbb{D}$ and $\mathbb{H}$ are convex and contain no complex affine lines. Then the following theorem states that the equality in Theorem \ref{Thm:fri_upper} always holds when $D=\mathbb{D}$ or $\mathbb{H}$. 

\begin{theorem}
\label{thm:fri:main}
Let $D=\mathbb{D}$ or $\mathbb{H}$.
Let $\Gamma$ be a torsion-free discrete group of isometries of $(D,\rho_D)$. Let $X = D \diagup \Gamma$ and $\pi$ be its quotient map. For any $z\in X$, let $w$ be any point in $D$ such that $\pi(w)=z$. Then we have
\[
    H^k_X (z) = \min\limits_{\gamma \in \Gamma \setminus \lbrace \mathrm{Id} \rbrace } \tanh \frac{ \rho_D (w,\gamma (w) )}{4} 
\]
\end{theorem}
Here, for $D=\mathbb{D}$ or $\mathbb{H}$, $\rho_D$ denotes the Poincar\'{e} metric on $D$ and it is well-known that $\rho_D=k_D=c_D$ when $D \subset \mathbb{D}$ is simply connected. In particular, Theorem \ref{thm:fri:main} allows us to obtain an explicit formula for $H^k_{A_r} (z)$ and $H^k_{\mathbb{D}^*} (z)$ in Theorem \ref{thm:fri_k_ar} and Theorem \ref{thm:fri_k_disk} respectively. We also obtain an explicit formulae for $H^c_{A_r} (z)$ and $H^c_{\mathbb{D}^*} (z)$ in Theorem \ref{thm:fri_c_ar} and Theorem \ref{thm:fri_c_disk} respectively. From these four theorems, we notice that for $d=k$ or $c$, if $p\in \partial A_r$ or $\partial \mathbb{D}$, then
\begin{equation*}
   \lim\limits_{z \to p} H^d_{A_r} (z) =
   \lim\limits_{z \to p} H^d_{\mathbb{D}^*} (z) =
   1
\end{equation*}
and if $p=0$, then $   \lim\limits_{z \to p} H^d_{\mathbb{D}^*} (z) =
   0$. These suggest us to study the boundary behaviour of certain Riemann surfaces.

A Riemann surface $X$ is said to be \textit{of regular type} if 
\begin{enumerate}
\item all connected component of the boundary $\partial X$ is either a Jordan curve or an isolated point, and;
\item all connected component of $\partial X$ are separated, i.e., for all connected component $\sigma$ of $\partial X$, there exists an open neighbourhood $U_\sigma$ of $\sigma$ such that $U_\sigma \cap \partial X = \sigma$.
\end{enumerate}
For example, $A_r$ and $\mathbb{D} \setminus \{ p_1,\dots, p_n \}$ are of regular type. Let $X$ be a hyperbolic regular type Riemann surface. By the Uniformization Theorem, we can assume that $X= \mathbb{D} \setminus \Gamma$ where $\Gamma$ is a torsion-free Fuchsian group (see for example Corollary 1.1.49 of \cite{aba}). Hence Theorem \ref{thm:fri:main} also allows us to explore the boundary behaviour of $H^k_X (z)$ when $n=1$ and $X$ is a Riemann surface of regular type. This is stated as the following theorem.

\begin{theorem}
\label{thm:fri:boundary_grow}
Let $X$ be a hyperbolic Riemann surface of regular type and $p \in \partial X$ be a boundary point. Let $\sigma \subset \partial X$ be the boundary component $p$ belongs to. 
\begin{enumerate}
    \item If $\sigma$ has only one point, then  we have 
  \[ \lim\limits_{z \to p} H^k_X (z) =0. \]
Hence,  $\lim\limits_{z \to p} S_X (z)= \lim\limits_{z \to p} H^c_X (z) =0$ when $X$ is a hyperbolic planar domain. 
    \item If $\sigma$ has more than one point, then we have 
\[ \lim\limits_{z \to p} H^k_X (z) =1. \]
\end{enumerate}
\end{theorem}

Note that the boundary behaviour of $S_X$ has been studied intensively, see for example, \cites{squ_infin,squ_def1,deng2016properties,fornaess2016estimate,fornaess2016non,joo2016boundary,kim2016uniform,nikolov2018behavior,zimmer2018gap,nikolov2020estimates,rong2020comparison,zimmer2019characterizing} and the survey \cite{deng2019holomorphic}. Notice that in these papers, the boundaries of the domains are assumed to satisfy certain smoothness conditions while there is no smoothness assumption on the boundaries of $X$ in Theorem \ref{thm:fri:boundary_grow}.

In views of Theorem \ref{thm:fri<iota}, it is natural to ask the following questions.\\

\noindent \textbf{Question 1.}
\textit{For which $c$-hyperbolic complex manifold $X$ can we have  
$H^c_X (z) \leq \iota^c_X (z) $ for all $z\in X$ ?}\\

\noindent \textbf{Question 2.}
\textit{Let $d=c$ or $k$. For which $d$-hyperbolic complex manifold $X$ the equality
\[ H^d_X (z) \equiv \iota^d_X (z)\]
holds? In particular, do we have $H^c_X (z) \equiv \iota^c_X (z)$ when $n=1$?} 

On the other hand, our studies on the boundary behaviour of $H^k_X (z)$ in Theorem \ref{thm:fri:boundary_grow} suggests the following question.

\begin{flushleft}
\textbf{Question 3.}
\textit{
Under the assumptions of Theorem \ref{thm:fri:boundary_grow} (case 2), do we have 
\[ \lim\limits_{z \to p} H^c_X (z) =1  \quad (\mbox{or even} \, \lim\limits_{z \to p} S_X (z) =1 \, \mbox{for planar} \, X)\, ?
\]
  }
\end{flushleft}

The rest of the paper goes as follows. We will prove the Theorems \ref{thm:fri<iota}, \ref{Thm:fri_upper} and \ref{thm:fri:main} in Section \ref{sect:result}. Then in Section \ref{sect:ex}, we will give explicit formulae for $H^k_{A_r}$,$H^c_{A_r}$ and also $H^k_{\mathbb{D}^*}$,$H^c_{\mathbb{D}^*}$  (Theorems \ref{thm:fri_k_ar},\ref{thm:fri_c_ar},
\ref{thm:fri_k_disk} and \ref{thm:fri_c_disk}). As applications to these explicit formulae,  we will calculate the injectivity radius of $A_r$ and $\mathbb{D}^*$ and  address some problems on $\frac{S_X (z)}{H^d_X (z)}$ studied by Rong and Yang  \cite{rong2020comparison} in Theorems \ref{thm:quo_Ar} and \ref{thm:quo_Disk}. Finally, we will also study the boundary behaviour of $H_X^k(z)$ for $k$-hyperbolic Riemann surface $X$ of regular type in Section \ref{sect:boundary_growth} (see Theorem \ref{thm:fri:boundary_grow}). 

Throughout the paper, we adopt the following notations.
\begin{itemize}
    \item $\mathcal{O}(X,Y)$ is the family of all holomorphic functions from $X \subset \mathbb{C}^n$ to $Y\subset \mathbb{C}^n$. 
    
    \item $\mathcal{U}(X,Y)$ is the family of all injective holomorphic functions from $X \subset \mathbb{C}^n$ to $Y\subset \mathbb{C}^n$. 

    \item $\mathbb{D}$ is the standard open unit disk in $\mathbb{C}$ and $\rho_\mathbb{D}$ is the Poincar\'{e} metric on $\mathbb{D}$ with density function $\frac{2}{1-|z|^2}.$ Note that for any $z \in \mathbb{D}$, we have $\rho_\mathbb{D} (0,z) = 2 \tanh^{-1} |z|.$ 
    
    \item $\mathbb{H}$ is the upper-half plane in $\mathbb{C}$ and $\rho_\mathbb{H}$ is the Poincar\'{e} metric on $\mathbb{H}$ with density function $\frac{1}{\mathrm{Im}(w)}.$ 
    
\item For any two points $z_1,z_2$ of a complex manifold $X$, the Carath\'{e}odory pseudo-distance $c_X (z_1,z_2)$ is defined to be 
\begin{equation*}
    c_X (z_1,z_2) := \sup \{ \rho_\mathbb{D} ( f(z_1),f(z_2) ) \: : \: f \in \mathcal{O} (X, \mathbb{D})\}. 
\end{equation*}

\item For any two points $z_1,z_2$ of a complex manifold $X$, the Kobayashi pseudo-distance $k_X (z_1,z_2)$ is defined to be
\[ 
k_X (z_1,z_2)
=
\inf \left\lbrace
\sum^n_{i=1} \rho_\mathbb{D} ( a_i, b_i)
\: : \: 
n \geq 1, 
f_i \in \mathcal{O} (\mathbb{D} ,X) 
\right\rbrace
\]
with $p_0,\ldots ,p_n \in X$, $a_1,\ldots, a_n,b_1, \ldots, b_n \in \mathbb{D}$, $f_1 (a_1) =p_0 =z_1$, $f_{i} (b_{i})=f_{i+1}( a_{i+1})$ for $i=1,\ldots,n-1$ and $f_n (b_n) =p_n =z_2$. As a remark, the Kobayashi pseudo-distance can be equivalently defined to be the largest pseudo-distance bounded above by the \textit{Lempert function} $L_X$, which is
\begin{equation*}
    L_X (z_1,z_2) = \inf \{ \rho_\mathbb{D}   (w_1,w_2) \: : \: f\in \mathcal{O}(\mathbb{D},X), f(w_1)=z_1, f(w_2)=z_2
    \}.
\end{equation*}

    \item Denote $B^k_X (z;r)$ (respectively, $B^c_X (z;r)$) the open Kobayashi ball (respectively, Carath\'{e}odory ball) in $X$ centered at $z$ with radius $r>0$, that is,
\begin{equation*}
    B^d_X (z;r)= 
    \{ w \in X \: : \: d_X (z,w) <r  
    \}
\end{equation*}
and $d$ can be either $k$ or $c$. Notice that $ B^k_{\mathbb{B}^n} (0;r) = \mathbb{B}^n (0; 2\tanh^{-1} r)$.
\end{itemize}


\section{Proofs of the main results}
\label{sect:result}

We first prove Theorem \ref{thm:fri<iota}. 

\begin{proof}[Proof of Theorem \ref{thm:fri<iota}.]
For part 1, we first consider the case $n=1$ and $d=k$. Then $X$ is a $k$-hyperbolic Riemann surface. For any open ball $B^k_X(z;r)$ of $X$, the inclusion map $\phi : B^k_X(z;r) \to X$ is holomorphic and hence distance-decreasing. Thus, $B^k_X (z;r)$ is $k$-hyperbolic. Suppose that $B^k_X (z;r)$ is simply connected. Then as $B^k_X (z;r)$ is $k$-hyperbolic, $B^k_X (z;r)$ is biholomorphic to $\mathbb{D}$ by the Uniformization Theorem (see Theorem 4.6.1. of \cite{book:hyper_geo}). Therefore, there exists a biholomorphic map $f: \mathbb{D} \to B^k_X (z;r)$ such that $B^k_X (z;r)= f (\mathbb{D})$ and hence $H^k_X (z) \geq \iota^k_X (z) $. We now consider the case when $n=1$ and $d=c$. Note that for any $z_1,z_2 \in X$, we have 
\begin{equation*}
c_X (z_1, z_2) \leq k_X (z_1, z_2).
\end{equation*}
Therefore, being a $c$-hyperbolic Riemann surface, $X$ is $k$-hyperbolic. It follows from the arguments for the previous case that if $B^c_X (z;r)$ is simply connected, then there exists a biholomorphic map $f: \mathbb{D} \to B^c_X (z;r)$ such that $B^c_X (z;r)= f (\mathbb{D})$ and hence $H^c_X (z) \geq \iota^c_X (z) $. This proves part 1.

For part 2, consider the case when $n \geq1$,  $d=k$ and $X = D \diagup \Gamma$, where $D \subset \mathbb{C}^n$ is a $k$-hyperbolic domain with all the open balls $B^k_X (z;r)$ of $X$ are simply connected and $\Gamma$ is a torsion-free discrete group of isometries of $D$.  Let $\pi$ be the quotient map of $X = D \diagup \Gamma$. That $H^k_X (z) \leq \iota_X^k (z)$ for all $z\in D$ will follow if $B^k_X (z;r)$ is simply connected whenever there exists an injective holomorphic function $f:\mathbb{B}^n \to X$ such that $B^k_X (z;r) \subset f(\mathbb{B}^n)$. We first show that $B^{k}_{X}(z;r)=\pi(B^{k}_{D} (w;r))$ whenever $w \in D$ and $\pi (w)=z$.

From the contraction property of Kobayashi metric, we have 
$$ k_{X} ( \pi(w_1) , \pi(w_2)) \leq k_{D} (w_1,w_2)$$
for any $w_1,w_2 \in D$. Hence, for any $r>0$ and $\zeta \in B^k_{D} (w;r)$, we have 
$$k_{X} ( z , \pi(\zeta)) \leq k_{D} (w,\zeta) < r$$ and thus

$$\pi( B^k_{D} (w;r) ) \subset B^k_{X}(z;r).$$
It is known that for any $z_1 , z_2 \in X$ and $w_1,w_2 \in D$ with $\pi(w_1)=z_1$ and $\pi(w_2)=z_2$, 
$$k_{X} ( z_1 , z_2) = \inf_{w_2 \in D}k_{D} (w_1,w_2)$$
(see for example, Theorem 3.2.8 of \cite{book:hyper_Kobayashi}). Then it follows that for any $\epsilon>0$,
$$B^k_{X}(z;r) \subset \pi( B^k_{D} (w;r+\epsilon)).$$
Hence, we have $B^{k}_{X}(z;r)=\pi(B^{k}_{D} (w;r))$.

If $\pi$ is injective on $B^{k}_{D} (w;r)$, then $B^{k}_{X}(z;r)=\pi(B^{k}_{D} (w;r))$ will be simply connected as we have assumed that any open ball $B^{k}_{D} (w;r)$ is simply connected.

If $\pi$ is not injective on $B^{k}_{D} (w;r)$, then there exist distinct $w_1,w_2 \in B^{k}_{D} (w;r)$ such that $\pi(w_1)=\pi(w_2)$. Let $\gamma$ be a curve in $B^{k}_{D} (w;r)$ joining $w_1$ to $w_2$. Then $\pi(\gamma)$ is a closed curve in $f(\mathbb{B}^n)$ as $\pi(B^{k}_{D} (w;r))=B^k_X (z;r) \subset f(\mathbb{B}^n)$.  

Since $\mathbb{B}^n$ is simply connected and $f$ is injective and holomorphic, we have $f(\mathbb{B}^n)$ is simply connected in $X$. Therefore, $\pi(\gamma)$ is homotopic to a point in $f(\mathbb{B}^n)$. If we lift this homotopy to a homotopy in $D$ through the covering map $\pi$, we can deduce that $w_1=w_2$ which is a contradiction.  

For part 3, note that $X$ is a $k$-hyperbolic Riemann surface. By part 1, we have $H^k_X (z) \geq \iota^k_X (z)$. 
By the Uniformization Theorem, $X$ is biholomorphic to $\mathbb{D} \diagup \Gamma$ where $\Gamma$ is a torsion-free discrete group of isometries of $(\mathbb{D},k_{\mathbb{D}})$ (cf. for example Corollary 1.1.49 of \cite{aba}). Since the Fridman functions and the injectivity radius functions are biholomorphic invariant, by part 2 we have $H^k_X (z) \leq \iota^k_X (z)$. The result follows.
 
\end{proof}

\begin{proof}[Proof of Theorem \ref{Thm:fri_upper}.] We first show that the expression 
\[ 
\min\limits_{\gamma \in \Gamma \setminus \lbrace \mathrm{Id} \rbrace } \tanh \frac{ d_{D} (w,\gamma (w) )}{4}
\]
is independent of the choice of $w$. To see this, consider another point $w'\in D$ such that $\pi(w')=z$. Then we have $w' = \sigma (w)$ for some $\sigma  \in \Gamma$. Because $\Gamma$ is a group of isometries of $d_D$, we have 
\begin{align*}
d_{D} (w',\gamma (w'))
&= d_{D} \left( \sigma^{-1} (w') ,\sigma^{-1} (\gamma (w')) \right) 
\\
&= d_{D} \left( w , \widetilde{ \gamma } (w) \right) 
\end{align*}
for some $\widetilde{ \gamma } = \sigma^{-1} \circ \gamma \circ \sigma$. Note that $\gamma \in \Gamma \setminus \lbrace \mathrm{Id} \rbrace$ if and only if $\widetilde{ \gamma } \in \Gamma \setminus \lbrace \mathrm{Id} \rbrace$. It follows that 
\[ 
\min\limits_{\gamma \in \Gamma \setminus \lbrace \mathrm{Id} \rbrace } \tanh \frac{ d_{D} (w,\gamma (w) )}{4}
=
\min\limits_{\gamma \in \Gamma \setminus \lbrace \mathrm{Id} \rbrace } \tanh \frac{ d_{D} (w',\gamma (w') )}{4}
\]
for any $w,w' \in D$ such that $\pi (w) =\pi (w')$.

Now we will establish some useful facts about the metric geometry of $(D,d_D)$. Because $D\subset \mathbb{C}^n$ is convex and contains no complex affine lines, we know that
for $d=k$ or $c$, $D$ is $d$-hyperbolic and $(D,d_{D})$ is a complete metric space in the sense that all closed balls of $(D,d_{D})$ are compact (see \cite{Barth1980} when $D$ is bounded and Theorem 1.1 and Lemma 3.1 of \cite{Bracci2009} when $D$ is unbounded).
In addition, $k_D=c_D$ (cf. Theorem 4.8.13 of \cite{book:hyper_Kobayashi} for bounded $D$ and Lemma 3.1 of \cite{Bracci2009} for the unbounded case).
Hence, by Theorem 3.2.1 of \cite{book:hyper_Kobayashi}, $k_D=c_D$ induces the Euclidean topology of $D$ which is locally compact. Therefore, the metric space $(D,d_D)$ is complete and locally compact and we can apply the Hopf-Rinow Theorem for length space (see for example, Proposition 3.7 of \cite{bridson2013}) to conclude that this metric space is a geodesic space which means any two points in $D$ are connected by a geodesic in $(D,d_D)$. Finally, by Corollary 4.8.3 of \cite{book:hyper_Kobayashi} and Proposition 3.2 of \cite{Bracci2009}, any $k$-ball (and hence $c$-ball) of $D$ is convex and hence path connected. Actually in general, for any complex manifold, any two points in a $k$-ball can be joined by a rectifiable curve (cf. Corollary 3.1.17 of \cite{book:hyper_Kobayashi}). Notice that $c$-balls can be disconnected (see \cite{jarnicki1992} or Chapter 2 of   \cite{jarnickipflug}).

Because $k_D=c_D$ and $H_X^c \le H_X^k$, we only need to show that (\ref{eq:fri_tanh}) holds for $d=k$.

Note that as a quotient of a $k$-hyperbolic complex manifold under a torsion-free discrete group of isometries, $X$ is also $k$-hyperbolic (cf. Theorem 3.2.8 of \cite{book:hyper_Kobayashi}). From the proof of  
part 2 of Theorem \ref{thm:fri<iota}, we also have $B^{k}_{X}(z;r)=\pi(B^{k}_{D} (w;r))$.

As $\Gamma$ is a torsion-free discrete group, the minimum of the set \[ 
\left\lbrace
k_{D} (w,\gamma (w) ) \: : \: \gamma \in \Gamma \setminus \lbrace \mathrm{Id} \rbrace 
\right\rbrace
\]
is attained and is greater than $0$ (see Theorem 5.3.4. of \cite{Book:Found_Hyper}). So if
\begin{equation*}
    r > \min\limits_{\gamma \in \Gamma \setminus \lbrace \mathrm{Id} \rbrace } \frac{ k_{D} (w,\gamma (w) )}{2}>0,
\end{equation*}
then there exists some $\widetilde{\gamma} \in \Gamma \setminus \lbrace \mathrm{Id} \rbrace$, such that 
\begin{equation*}
    r > \frac{ k_{D} (w,\widetilde{\gamma}  (w) )}{2} >0.
\end{equation*}
Consequently, on the geodesic arc from $w$ to $\widetilde{\gamma}(w)$, there exists some $u$ such that $u \in B^k_{D} (w;r)\cap B^k_{D} (\widetilde{\gamma}(w);r)$. Recall that an element $\gamma$ in $\Gamma$ is said to be \textit{elliptic} if $\gamma$ fixes an interior point of $D$. Because $\Gamma$ is discrete, any elliptic element $\gamma$ in $\Gamma$ has finite order (see the remark after Theorem 5.4.1 in \cite{Book:Found_Hyper}). But since $\Gamma$ is torsion free, we have $\Gamma$ contains no elliptic elements and hence $\gamma$ has no fixed points in the interior for any $\gamma \in \Gamma \setminus \lbrace \mathrm{Id} \rbrace$. In particular, we have $u \neq \widetilde{\gamma}^{-1} (u)$. 

As $\widetilde{\gamma}^{-1}$ is an isometry of $(D,k_D)$, we have
\[ k_D ( \widetilde{\gamma}^{-1} (u) , w )= k_D (u, \widetilde{\gamma} (w)  )<r.\]
This implies that $\widetilde{\gamma}^{-1} (u) \in B^k_{D} (w;r)$. Recall that any open ball $B^k_D (w;r)$ is path connected. Let $L_1$ be a simple path in $B^k_D (w;r)$ joining $u$ to $w$ and $L_2$ be a simple path in $B^k_D (w;r)$ from $w$ to $\widetilde{\gamma}^{-1} (u)$. Define $l_1=\pi (L_1)$ and $l_2=\pi (L_2)$. Since $L_1,L_2 \subset B^k_{D} (w;r)$, we have $l_1,l_2 \subset B^k_X (z;r)=\pi(B^{k}_{D} (w;r))$. Moreover, $l_1$ and $l_2$ have the same end points, namely, $z_1= \pi (w)$ and $z_2=\pi (u)= \pi ( \widetilde{\gamma}^{-1} (u))$. 

Assume to the contrary that $H^k_X (z) > \min\limits_{\gamma \in \Gamma \setminus \lbrace \mathrm{Id} \rbrace } \tanh \frac{ k_{D} (w,\gamma (w) )}{4}$. Let $r>0$ such that
\[
 H^k_X (z)>  \tanh \frac{r}{2} > \min\limits_{\gamma \in \Gamma \setminus \lbrace \mathrm{Id} \rbrace } \frac{ k_{D} (w,\widetilde{\gamma}  (w) )}{4} .
\]
Then $B^k_X( z;r)\subset f(\mathbb{B}^n)$ for some $f \in \mathcal{U} ( \mathbb{B}^n , X )$. Because $\Gamma$ is a torsion-free discrete group of isometries of $(D,k_D)$, $\pi$ is a regular covering map of $X$ (see for example Theorem 81.5 in \cite{book:topo}). Since $L_1$ and $L_2$ have the same starting point $w$  but different end points $u$ and $\widetilde{\gamma}^{-1} (u)$, $L_1$ and $L_2$ are not homotopic in $D$. By Theorem 54.3 in \cite{book:topo}, $l_1$ and $l_2$ are not homotopic in $X$. It follows that 
$l_1$ and $l_2$ are not homotopic in $f(\mathbb{B}^n)$, which is a contradiction as $f(\mathbb{B}^n)$ is simply-connected. 
Consequently,
\begin{equation*}
    H_X^c (z) \leq H_X^k (z) \leq \min\limits_{\gamma \in \Gamma \setminus \lbrace \mathrm{Id} \rbrace } \tanh \frac{ k_{D} (w,\gamma (w) )}{4}=\min\limits_{\gamma \in \Gamma \setminus \lbrace \mathrm{Id} \rbrace } \tanh \frac{ c_{D} (w,\gamma (w) )}{4}
\end{equation*}
and we complete the proof.
 
\end{proof}

\begin{proof}[Proof of Theorem \ref{thm:fri:main}.]
In this setting, we have $k_D=\rho_D$, $k_X=\rho_X$ and 
\begin{equation*}
    \rho_X (z, t) = \min \{ \rho_D (w,s) \: : \:  \pi (s)= t \}
\end{equation*}
(cf. Chapter 7 of \cite{book:hyper_geo}). This implies that implies that $ \pi( B^k_{D} (w;r) ) = B^k_{X}(z;r) $ (see also the proof of part $2$ of Theorem \ref{thm:fri<iota}). By Theorem \ref{thm:fri<iota}, $H^k_X (z) = \iota^k_X (z)$ when $n=1$. Thus it suffices to prove that for any 
\begin{equation*}
    r \leq \min\limits_{\gamma \in \Gamma \setminus \lbrace \mathrm{Id} \rbrace } \frac{ \rho_D (w,\gamma (w) )}{2} ,
\end{equation*}
$B^k_{X}(z;r)$ is simply-connected in $X$. Assume to the contrary that it is not simply-connected. Then there exists paths $L$, $L'$, both start at $z$ and end at some point $t \in B^k_{X}(z;r)$, but not homotopic in $B^k_{X}(z;r)$. By the path lifting property, $L$ lifts up to a path $\widetilde{L} \subset B^k_{D} (w;r)$ starting at $w$ and ending at some point $s \in B^k_{D} (w;r)$ whereas $L'$ lifts up to a path $\widetilde{L}' \subset B^k_{D} (w;r)$ starting at $w$ and ending at some point $s' \in B^k_{D} (w;r)$. Note that $s' \neq s$ because $B^k_D (w;r)$ is simply connected by Remark \ref{ball}. Since $\pi (s) =\pi (s') =t$, there exists some $\widetilde{\gamma} \in \Gamma \setminus \lbrace \mathrm{Id} \rbrace $ such that $s' = \widetilde{\gamma} (s)$. But then we have 
\begin{align*}
 \rho_{D} (w , \widetilde{\gamma} (w))
 &\leq \rho_{D} (w , \widetilde{\gamma} (s)) + \rho_{D} (\widetilde{\gamma} (s),\widetilde{\gamma} (w))
 \\
 &= \rho_{D} (w , \widetilde{\gamma} (s)) + \rho_{D} (w,s)
 \\
 &<
2r
\\
&\leq 
\min\limits_{ \gamma \in \Gamma \setminus \lbrace \mathrm{Id} \rbrace } \rho_{D} (w, \gamma (w)),
\end{align*}
which is a contradiction. Hence $B^{k}_{X}(z;r)=\pi(B^{k}_{D} (w;r))$ is simply-connected for any  $r \leq \min\limits_{\gamma \in \Gamma \setminus \lbrace \mathrm{Id} \rbrace } \frac{ \rho_D (w,\gamma (w) )}{2}$
and thus  
\begin{align*}
    H^k_X (z) &= \min\limits_{\gamma \in \Gamma \setminus \lbrace \mathrm{Id} \rbrace } \tanh \frac{ \rho_D (w,\gamma (w) )}{4} .
\end{align*}
 
\end{proof}


\section{Some explicit Fridman functions}
\label{sect:ex}

In this section, we will make use of Theorem \ref{thm:fri:main} to do some computations. 

\subsection{Example 1: Fridman functions for an annulus}
\label{ch:fri.Sec:cal:ar}

Let $A_r = \lbrace z  \in \mathbb{C}  \: : \: r<|z|<1 \rbrace$. In our previous paper \cite{ourpaper_squeezing_annulus}, we have proven that the explicit form of $S_{A_r}(z)$ is given by
\[ S_{A_r}(z) = \max \left\lbrace |z|, \frac{r}{|z|} \right\rbrace. \]
As an analog, we will give the explicit expression of $H^d_{A_r} (z)$ for both $d=k$ and $c$.

\begin{theorem}
\label{thm:fri_k_ar}
Fix $r \in (0,1)$. For any $z \in A_r$, we have
\begin{equation*}
    H^k_{A_r} (z) =
\dfrac{ \sqrt{\left(1-\lambda^2 \right)^2 + 4 \lambda^2 \sin^2 \theta (z) } - 2 \lambda \sin \theta (z)}{1-\lambda^2}
\end{equation*}
where 
$\lambda = \exp\left( \frac{\pi^2}{\ln r} \right)$
and 
$ \theta (z)=  \frac{\pi \ln |z| }{\ln r} $.
\end{theorem}

\begin{proof}
Define $ \lambda = \exp \left( \frac{ \pi^2 }{\ln r} \right) $ and $\gamma: \mathbb{H} \to \mathbb{H}$ such that $ \gamma (w)= \lambda^2 w$ for any $w\in \mathbb{H}$. Let $\Gamma = \langle \gamma \rangle$ be the group generated by $\gamma$. Then, $\Gamma$ is a torsion-free discrete group of isometries of $\mathbb{H}$. Let $\pi: \mathbb{H} \to \mathbb{H}\diagup \Gamma$ be the quotient map. 

Define a function $\Phi : \mathbb{H} \to A_r$ such that $\Phi (w) = \exp \left( i \left( \frac{ - \ln r}{\pi} \right) 	\ln w \right)$ for any $w \in \mathbb{H}$, where we choose the branch of logarithm so that $\mathrm{Im} (\ln (w)) \in (0,2\pi)$.
Then $\Phi (\gamma (w) ) = \Phi (w)$ for any $w \in \mathbb{H}$ and for any $\gamma \in \Gamma$. Hence $\Phi$ descends to a biholomorphic map $\phi : \mathbb{H} \diagup \Gamma \to A_r$ and we have the following commutative diagram
\begin{equation*}
    \begin{tikzpicture}[node distance=4cm, auto]
  \node (H) {$\mathbb{H}$};
  \node (X) [below of=H] {$\mathbb{H} \diagup \Gamma$};
  \node (A) [right of=X] {$A_r$};
  \draw[->] (H) to node {$\Phi$} (A);
  \draw[->] (H) to node [swap] {$\pi$} (X);
  \draw[->] (X) to node [swap] {$\phi$} (A);
\end{tikzpicture}
\end{equation*}
Because the Fridman function is a biholomorphic invariant, we have 
\begin{equation*}
 H^k_{A_r} (z) =  H^k_{\mathbb{H} \diagup \Gamma} ( \zeta)
\end{equation*}
where $\zeta$ is the point in $\mathbb{H} \diagup \Gamma$ such that $\phi (\zeta)= z$. Theorem \ref{thm:fri:main} states that
\begin{equation*}
 H^k_{\mathbb{H} \diagup \Gamma} (\zeta) = \min\limits_{n \in \mathbb{Z}\setminus \{ 0\} } \tanh \frac{ \rho_\mathbb{H} (w,\gamma^n (w) )}{4},
\end{equation*}
for some $w \in \mathbb{H}$ such that $\pi (w) = \zeta$. Then the commutative diagram implies that
\begin{equation*}
 H^k_{A_r} (z) = \min\limits_{n \in \mathbb{Z}\setminus \{ 0\} } \tanh \frac{ \rho_\mathbb{H} (w,\gamma^n (w) )}{4},
\end{equation*}
for some $w \in \mathbb{H}$ such that $\Phi (w) =z$.
Note that for any $w_1,w_2 \in \mathbb{H}$, we have
\begin{equation*}
    \rho_\mathbb{H} ( w_1,w_2)
    = 2 \sinh^{-1} \dfrac{|w_1-w_2|}{2 \sqrt{\mathrm{Im}\left( w_1 \right)\mathrm{Im}\left( w_2 \right)} }
\end{equation*}
(see for example Theorem 7.2.1 of \cite{book:ddg}.)
Write $w=\rho e^{i\theta }$ for some $\rho>0$ and $\theta \in \mathbb{R}$. Since $\gamma (w) =\lambda^2 w$ with $\lambda \in (0,1)$, we get
\begin{align*}
\rho_\mathbb{H} (w,\gamma^n (w) )
&=
 2 \sinh^{-1} \dfrac{|w- \lambda^{2n} w|}{2 \sqrt{\mathrm{Im}\left( w \right)\mathrm{Im} \left( \lambda^{2n} w \right)} }
 \\
&=
 2 \sinh^{-1} 
  \dfrac{(1- \lambda^{2n}) \: \rho }{2 \sqrt{\rho^2 \lambda^{2n} \sin \theta } }
 \\
&=
 2 \sinh^{-1} \left( 
\dfrac{\lambda^{-n} - \lambda^n}{2\sin \theta}
 \right).
\end{align*}
Since $\tanh$ and $\sinh^{-1}$ are increasing functions on $\mathbb{R}$ and the value of $\lambda^{-n} - \lambda^n$ increases as $n$ increases, the minimum of $\tanh \frac{ \rho_\mathbb{H} (w,\gamma^n (w) )}{4}$ is attained when $n=1$. Thus, we have
\begin{equation*} H^k_{A_r} (z) =  \tanh \frac{ \rho_\mathbb{H} (w, \lambda^2 w )}{4}
\end{equation*}
for some $w \in \mathbb{H}$ such that $\Phi (w)=z$. Since $\Phi (w)=z$, we have
\[ 
w = \exp \left( \frac{i \pi \ln z}{\ln r} \right)
\]
for some branches of $\ln$. Then direct calculation yields 
\begin{equation*}
    \rho_\mathbb{H} (w,\gamma (w))
    =
    2 \sinh^{-1} \dfrac{1-\lambda^2 }{2 \lambda \sin \theta (z) }
\end{equation*}
where $ \lambda = \exp \left( \frac{ \pi^2 }{\ln r} \right) $ as defined above and 
$ \theta(z)= \frac{\pi \ln |z| }{\ln r} $. Since  
\begin{equation*}
    \tanh \dfrac{x}{2}
    =
    \dfrac{\sqrt{1+\sinh^2 x}-1}{\sinh x}
\end{equation*}
for any $x \in \mathbb{R}$, the result follows by direct substitution.
 
\end{proof}

By the expression given in Theorem \ref{thm:fri_k_ar}, elementary calculus shows that minimum of $H^k_{A_r}(z)$ is attained when $\theta = \frac{\pi}{2}$, i.e., when $|z|=\sqrt{r}$. We have the following corollary.

\begin{corollary}
\[ \iota_{A_r} = \dfrac{1-\lambda}{1+\lambda} \]
where 
$\lambda = \exp\left( \frac{\pi^2}{\ln r} \right)$.
\end{corollary}

\begin{remark}
This can also be obtained by a special case of Theorem 2.3 of Sugawa \cite{inject2}.
\end{remark}

In view of inequality (\ref{eq:fri_order}), we would also like to determine the explicit form of $H^c_{A_r}(z)$. We first recall that 
\[ c_{A_r}(z_1,z_2) := \sup \{ \rho_\mathbb{D} (f(z_1),f(z_2)) \: : \: f \in \mathcal{O} (A_r , \mathbb{D}) \}. \] 
By Grunsky \cites{grunsky1,grunsky2} and Ahlfors \cite{paper_ahlfor_cara}, the maximizing function $f$ is a ramified double cover of $\mathbb{D}$, unique up to postcompositing a rotation. Then Simha \cite{paper:Simha} gives an explicit formula for this maximizing function and hence the Carath\'{e}odory metric of the annulus in the complex plane. For points $z_1,z_2 \in A_r$ with $z_1>0$ and $z_2<0$, he showed in \cite{paper:Simha} that
\[
\tanh \dfrac{1}{2} c_{A_r}(z_1,z_2)
=F_{z_1}(z_2)
\]
where
\begin{equation}
\label{eq:fri_cara1}
F_{z_1}(z_2)
=
\left( \dfrac{-r}{z_2} \right) \left( 1- \dfrac{z_2}{z_1} \right) \left( 1- \dfrac{z_1z_2}{r} \right)
Q(z_1,z_2,r)
\end{equation}
and
\begin{align*}
    & Q(z_1,z_2,r) \\
= &
\prod\limits_{n=1}^{\infty}
\dfrac{(1-z_1^{-1}z_2 r^{2n})(1-z_1z_2^{-1} r^{2n})(1-z_1z_2 r^{2n-1})(1-z_1^{-1}z_2^{-1} r^{2n+1})}{(1-z_1^{-1}z_2 r^{2n-1})(1-z_1 z_2^{-1} r^{2n-1})(1-z_1 z_2 r^{2n-2})(1-z_1^{-1}z_2^{-1}r^{2n})}
\end{align*}
is a ramified double cover of $\mathbb{D}$ with zeros $z_1$ and $\frac{r}{z_1}$.

\begin{remark}
Let $\omega (z,y)$ be the Schottky-klein prime function on annulus $A_r$, which can be expressed as 
\begin{equation*}
\label{eq:skpf_ar}
    \omega (z, \zeta)
=
(z- \zeta) \prod\limits_{n=1}^\infty \dfrac{(r^{2n}z-\zeta )(r^{2n} \zeta -z )}{((r^{2n}z-z )(r^{2n} \zeta -\zeta )}.
\end{equation*}
Then we can write 
\begin{equation}
\label{eq:fri_cara2}
F_{z_1}(z_2)
=
\dfrac{1}{ r z_1} \left( \dfrac{\omega \left( z_1 , z_2 \right) \omega \left( z_1 , r z_2^{-1} \right) }{ \omega \left( z_1 , z_2^{-1} \right) \omega \left( z_1 ,z_2 r^{-1} \right) } \right).
\end{equation}
For more information about the Schottky-klein prime function, see for example \cite{book_crowdy2020}. 
\end{remark}

The following lemma is a consequence of the formulae for $c_{A_r}$.
\begin{lemma} 
\label{lem:fri:cara_min} 
Fix any $r \in (0,1)$. For any $r<z_1<1$ and $-1<z_2<-r$, we have the following.
\begin{enumerate}
    \item $c_{A_r}(z_1 , z_2 ) = c_{A_r} \left(z_1 , r z_2^{-1}\right).$
    \item Suppose that $z_1$ fixed and take $c_{A_r} (z_1 , z_2 )$ as a function of $z_2$ as $z_2$ varies in $(-1,-r)$. Then the minimum value of $c_{A_r} (z_1 , z_2 )$ is attained when $z_2=-\sqrt{r}$ (independent to $z_1$). 
\end{enumerate}
\end{lemma}

\begin{proof}
Replacing $z_2$ by $rz_2^{-1}$ in equation (\ref{eq:fri_cara1}), a direct calculation yields part 1 of the lemma. 

We now prove part 2 of the lemma. Suppose to the contrary that the minimum is attained at some point $\zeta \neq -\sqrt{r}$. By part 1 of this lemma, we may assume without loss of generality that $\zeta \in ( - \sqrt{r} ,-r)$. Note that $c_{A_r}(z_1 , z_2 )$ tends to infinity when $z_2$ approaches $-r$. Since the Carath\'{e}odory distance is continuous, by the intermediate value theorem there exists some point $y \in ( \zeta ,-r)$ such that $c_{A_r} (z_1 ,y)=c_{A_r} (z_1,-\sqrt{r})$. Then part 1 of this lemma implies that $y,\sqrt{r}$ and $\frac{r}{y}$ are three distinct points in $A_r$ such that 
\[
f_{z_1} (y)=f_{z_1} (-\sqrt{r} )=f_{z_1} \left( \frac{r}{y} \right)
\]
This contradicts with the fact that $f_{z_1}$ is a double cover of $\mathbb{D}$. Thus the minimum value of $c(z_1,z_2)$, as a function of $z_2 \in (-1,-r)$ with $z_1$ fixed, is attained for $z_2=-\sqrt{r}$. 
 
\end{proof}

Now we are ready to give the precise formula for $H^c_{A_r}(z)$ for any $r \in (0,1)$.

\begin{theorem}
\label{thm:fri_c_ar}
For any $z\in A_r$, we have
\begin{align*}
H^c_{A_r}(z)
& =
\sqrt{r} \left( 1 + \frac{\sqrt{r}}{|z|} \right) \left( 1+ \dfrac{|z|}{\sqrt{r}} \right)
\left( 
\prod\limits_{n=1}^{\infty}
\dfrac{(1+|z| r^{2n-1/2})( |z| + r^{2n+1/2})}{(1+|z| r^{2n-3/2})(|z | + r^{2n-1/2})}
\right)^2 
.
\end{align*}
\end{theorem}

\begin{remark}
\label{rmk:fri_c_ar}
Using equation (\ref{eq:fri_cara2}), we can write 
\begin{equation}
\label{Rmk:HcAR}
H^c_{A_r}(z)
=
\dfrac{1}{ r |z| } \left( \dfrac{\omega \left( |z| , -r^{1/2} \right)}{ \omega \left( |z| , -r^{-1/2} \right)} \right)^2.
\end{equation}
\end{remark}

\begin{proof}
Pre-composing rotation if necessary, we can assume $z>0$. Let 
\begin{equation*}
\rho^* = \inf\limits_{\zeta \in (-1 ,r)}  c_{A_r} (z,\zeta).
\end{equation*}
By Lemma \ref{lem:fri:cara_min}, this is attained when $\zeta=-\sqrt{r}$. Hence,
\begin{equation*}
\rho^* = c_{A_{r}}(z, -\sqrt{r}).
\end{equation*}
Then for any $\rho >0$ such that $\rho < \rho^*$, we have 
$ B^c_{A_r} (z ;\rho ) \subset A_r \setminus (-1, -r)$. Since $A_r \setminus (-1,-r)$ is simply connected, by the Riemann mapping theorem, there exists a biholomorphic map $f: \mathbb{D} \to A_r \setminus (-1, -r)$. Hence, by definition, 
\begin{equation*}
H^c_{A_r} (z) \geq \tanh \dfrac{1}{2} \rho^*
\end{equation*}

It suffices to show that $H^c_{A_r} (z) \leq \tanh \dfrac{1}{2} \rho^*$, or equivalently, for any $\rho > \rho^*$, $B^c_{A_r} (z;\rho) \not\subset f(\mathbb{D})$ for any $f \in \mathcal{U}(\mathbb{D},A_r)$. Suppose to the contrary that there exists $\rho > \rho^*$ such that $B^c_{A_r} (z;\rho) \subset f(\mathbb{D})$ for some $f \in \mathcal{U}(\mathbb{D},A_r)$. By construction of $\rho^*$, there exists $\zeta \in (-1,-r)$ such that $\zeta \in B^c_{A_r} (z; \rho)$. By \cite{c_con}, $B^c_{A_r} (z ;\rho)$ is connected. Thus, $B^c_{A_r} (z ; \rho) \subset \mathbb{C}$ is path-connected. Hence, there exists a path $\gamma_1$ in $B^c_{A_r} (z; \rho)$ connecting $z$ and $\zeta$. By reflection symmetry of $A_r$, and hence $B^c_{A_r} (z; \rho)$, we can assume $\gamma_1$ lies on the closed upper half-plane. Its reflection $\gamma_2$ along the real axis, is a path in $B^c_{A_r} (z; \rho)$, lying on the closed lower half-plane, connecting $z$ and $\zeta$. Then $\gamma_1$ and $\gamma_2$ together induce a closed curve $\gamma$ in $B^c_{A_r} (z; \rho)$. But then $\gamma$ is a closed curve in $f(\mathbb{D})$. Note that $\gamma$ is not null-homotopic in $A_r$. This is a contradiction because $f(\mathbb{D})$ is simply-connected. Consequently, we must have $H^c_{A_r} (z) \leq \tanh \dfrac{1}{2} \rho^*$. Substituting $\zeta=-\sqrt{r}$ into equation (\ref{eq:fri_cara1}), the result follows. 
 
\end{proof}

\begin{figure}
    \centering
    \includegraphics[height=10cm]{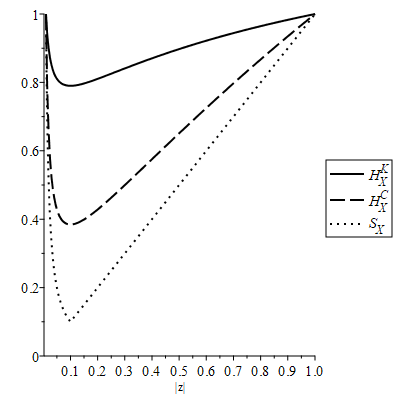}
    \caption{Plots of $H^k_{A_{0.01}}$ (solid), $H^c_{A_{0.01}}$ (dash) and $S_{A_{0.01}}$ (dot).}
\label{fig:plot}
\end{figure}

 If we take $r=0.01$ in Theorem \ref{thm:fri_k_ar} and Theorem \ref{thm:fri_c_ar}, we know that inequalities in (\ref{eq:fri_order}) may be strict (see Figure \ref{fig:plot}). 

\subsection{Example 2: Fridman functions for the punctured disk}
\label{ch:fri.Sec:cal:disk}

Denote $ \mathbb{D^*}=\{ z \in \mathbb{C} \: : \: 0<|z|<1 \} $ 
the punctured unit disk in $\mathbb{C}$. Similar to Section \ref{ch:fri.Sec:cal:ar}, we have the following theorems. 
 
\begin{theorem}
\label{thm:fri_k_disk} 
For any $z \in \mathbb{D}^*$, we have
\begin{equation*}
    H^k_{\mathbb{D}^*} (z) = 
    \dfrac{\log |z| + \sqrt{ (\log|z|)^2 + \pi^2} }{\pi}.
\end{equation*}
\end{theorem}

\begin{proof}
Define $\gamma: \mathbb{H} \to \mathbb{H}$ by $ \gamma (w)= w+ 2\pi$ for any $w\in \mathbb{H}$. Let $\Gamma = \langle \gamma \rangle$ be the group generated by $\gamma$. Then, $\Gamma$ is a torsion-free discrete group of isometries of $\mathbb{H}$. Let $\pi: \mathbb{H} \to \mathbb{H}\diagup \Gamma$ be the quotient map. 
Define a function $\Phi : \mathbb{H} \to \mathbb{D}^*$ such that $\Phi (w) = e^{iw}$ for any $w \in \mathbb{H}$. Then $\Phi (\gamma (w) ) = \Phi (w)$ for any $w \in \mathbb{H}$ and for any $\gamma \in \Gamma$. Hence, $\Phi$ descends to a biholomorphic map $\phi : \mathbb{H} \diagup \Gamma \to \mathbb{D}^*$ and we have the following commutative diagram
\begin{equation*}
    \begin{tikzpicture}[node distance=4cm, auto]
  \node (H) {$\mathbb{H}$};
  \node (X) [below of=H] {$\mathbb{H} \diagup \Gamma$};
  \node (A) [right of=X] {$\mathbb{D}^*$};
  \draw[->] (H) to node {$\Phi$} (A);
  \draw[->] (H) to node [swap] {$\pi$} (X);
  \draw[->] (X) to node [swap] {$\phi$} (A);
\end{tikzpicture}
\end{equation*}

Because the Fridman function is a biholomorphic invariant, we have 
\begin{equation*}
 H^k_{\mathbb{D}^*} (z) =  H^k_{\mathbb{H} \diagup \Gamma} ( \zeta)
\end{equation*}
where $\zeta$ is the point in $\mathbb{H} \diagup \Gamma$ such that $\phi (\zeta)= z$. Theorem \ref{thm:fri:main} states that
\begin{equation*}
 H^k_{\mathbb{H} \diagup \Gamma} (\zeta) = \min\limits_{n \in \mathbb{Z}\setminus \{ 0\} } \tanh \frac{ \rho_\mathbb{H} (w,\gamma^n (w) )}{4},
\end{equation*}
for some $w \in \mathbb{H}$ such that $\pi (w) = \zeta$. Then the commutative diagram implies that
\begin{equation*}
 H^k_{\mathbb{D}^*} (z) = \min\limits_{n \in \mathbb{Z}\setminus \{ 0\} } \tanh \frac{ \rho_\mathbb{H} (w,\gamma^n (w) )}{4},
\end{equation*}
for some $w$ such that $w = \Phi (z)$. Note that for any $w_1,w_2 \in \mathbb{H}$ with $\mathrm{Im} (w_1) = \mathrm{Im} (w_2)$, we have
\begin{equation*}
    \rho_\mathbb{H} ( w_1,w_2)
    = 2 \sinh^{-1} \dfrac{\left| \mathrm{Re}(w_1)- \mathrm{Re}(w_2) \right|}{2\mathrm{Im} (w_1)  }.
\end{equation*}
(see Theorem 7.2.1 of \cite{book:ddg}). Write $w=x+iy$ for some $x \in \mathbb{R}$ and $y >0$. Since $\gamma (w) = w +2\pi$ we get
\begin{equation*}
    \rho_\mathbb{H} ( w, \gamma^n (w) )
    = 2 \sinh^{-1} \dfrac{n \pi}{y }.
\end{equation*}
Because $\tanh$ and $\sinh^{-1}$ are increasing functions on $\mathbb{R}$, 
the minimum of $\tanh \frac{ \rho_\mathbb{H} (w,\gamma^n (w) )}{4}$ is attained when $n=1$. Thus, we have
\begin{equation*} H^k_{\mathbb{D}^*} (z) =  \tanh \frac{ \rho_\mathbb{H} (w,w+2\pi )}{4}
\end{equation*}
for some $w \in \mathbb{H}$ such that $\Phi (w)=z$. Since $\Phi (w)=z$, we have $w =  - i \ln z $
for some branches of $\ln$. As
\begin{equation*}
    \tanh \dfrac{x}{2}
    =
    \dfrac{\sqrt{1+\sinh^2 x}-1}{\sinh x}
\end{equation*}
for any $x \in \mathbb{R}$, the result follows by direction substitution.
 
\end{proof}

\begin{corollary}
\[ \iota_{\mathbb{D}^*} = 0 \]
\end{corollary}

\begin{remark}
This can also be obtained by a special case of Theorem 2.3 of \cite{inject2}
\end{remark}

\begin{proof}
Elementary calculus shows that $H^k_{\mathbb{D}^*} (z)$ is strictly increasing when $|z|$ increases. Hence 
$ \iota_{\mathbb{D}^*}
=
\lim\limits_{|z| \to 0}
H^k_{\mathbb{D}^*} (z)
= 0. $
 
\end{proof}

\begin{remark}
The formula we obtained for $H^k_{\mathbb{D}^*}(z)$ verifies Lemma 2.2 of \cite{mahajan2019comparison}.
\end{remark}

\begin{theorem}
\label{theorem}
\label{thm:fri_c_disk}
For any $z \in \mathbb{D}^*$, we have
\[
H^c_{\mathbb{D}^*}(z)= |z|.
\]
\end{theorem}

\begin{proof}
For any $f \in \mathcal{O}(\mathbb{D}^*,\mathbb{D})$, by the removable singularity theorem it extends to a function $\widetilde{f} \in \mathcal{O}(\mathbb{D},\mathbb{D})$ whereas any $\widetilde{f} \in \mathcal{O}(\mathbb{D},\mathbb{D})$ naturally defines  a function $f \in \mathcal{O}(\mathbb{D}^*,\mathbb{D})$. Thus we have
\[
c_{\mathbb{D}^*}(z_1,z_2)
=
c_{\mathbb{D}}(z_1,z_2).
\]
It also follows from the Schwarz Lemma that the $c_{\mathbb{D}}(z_1,z_2) = \rho_{\mathbb{D}}(z_1,z_2)$. Then for any $z \in \mathbb{D}^*$, we have $B^c_{\mathbb{D}^*}(z;2\tanh^{-1} |z|)$ is the largest possible Carath\'{e}odory ball centered at $z$ lying inside $\mathbb{D}^*$. Assume without loss of generality that $z>0$.
Observe that $z' \notin B^c_{\mathbb{D}^*}(z; 2\tanh^{-1} |z| )$ for any $-1<z'<0$. Consequently, we can conclude that $B^c_{\mathbb{D}^*}(z;2\tanh^{-1} |z|)$ is simply connected and hence the result follows. 
 
\end{proof}

\begin{remark}
Since
\[
|z| \leq \dfrac{\log |z| + \sqrt{(\log |z|)^2+1}}{\pi}
\]
for any $0<|z|<1$ (see the proof for Theorem \ref{thm:quo_Disk}), Theorem \ref{thm:fri_k_disk} and Theorem \ref{thm:fri_c_disk} verify Inequality (\ref{eq:fri_order}) when $X= \mathbb{D}^*$.
\end{remark}

\begin{remark}
Since
\[
S_{\mathbb{D}^*} (z) = |z|
\]
for any $0<|z|<1$, we have  $S_{\mathbb{D}^*} (z)  = H^c_{\mathbb{D}^*}(z)$. In fact, the same argument in Theorem \ref{thm:fri_c_disk} shows that, for $X=\mathbb{D} \setminus \{ p_1, \ldots, p_n \}$, we have
\[
H^c_X (z) =
\min\limits_{i=1,\dots,n}
\left\lbrace
\left| \dfrac{z-p_i}{1-z \overline{p_i}} \right|
\right\rbrace
\]
and hence $H^c_X (z) = S_X (z)$ for $X=\mathbb{D} \setminus \{ p_1, \ldots, p_n \}$.
\end{remark}

\subsection{On the comparison of the Fridman function and squeezing function}

Let $X$ be a bounded domain in $\mathbb{C}^n$. 
In \cite{rong2020comparison}, Rong and Yang introduced the \textit{quotient invariant} $m^d_X (z)= \frac{S_X (z)}{H^d_X (z)}$ for all $z \in X$. Since $S_X$ and $H^d_X$ are biholomorphic invariant, $m^d_X$ is also a biholomorphic invariant. From (\ref{eq:fri_order}), we know that $m^d_X (z) \le 1$ for all $z \in X$ and Rong and Yang asked for which $X$ and $z_0\in X$ one can have $m^d_X (z_0) = 1$. Apply Theorem \ref{thm:fri_c_ar} and its remark and the fact that $S_{A_r} (z) = \max\left\lbrace |z|,\frac{|z|}{r} \right\rbrace$, we have the following theorem which generalizes corollary 6 of \cite{rong2020comparison}. 

\begin{theorem}
\label{thm:quo_Ar}
For $d=k$ or $c$ and for any $z \in A_r$, we have
\[ m^d_{A_r} (z)  < 1 \]
and when $z \to p$ for any $p \in \partial A_r$, we have 
\[ \lim\limits_{z \to p } m^d_{A_r} (z)  = 1. \]
\end{theorem}

\begin{proof}
For $x \in (r,1)$, define
\[
f(x) := 
\dfrac{1}{ r x } \left( \dfrac{\omega \left( x , -r^{1/2} \right)}{ \omega \left( x , -r^{-1/2} \right)} \right)^2
- \dfrac{r}{x}
\]
Then we have 
\begin{align*}
    f(x) 
    &= 
\dfrac{1}{x}
\left[
\left( \dfrac{\omega \left( x , -r^{1/2} \right)}{ \sqrt{r} \omega \left( x , -r^{-1/2} \right)} \right)^2
-r
\right]
\\
    &= 
\dfrac{1}{x}
\left(
 \dfrac{\omega \left( x , -r^{1/2} \right)}{ \sqrt{r} \omega \left( x , -r^{-1/2} \right)} 
- \sqrt{r}
\right)
\left(
 \dfrac{\omega \left( x , -r^{1/2} \right)}{ \sqrt{r} \omega \left( x , -r^{-1/2} \right)} 
+ \sqrt{r}
\right).
\end{align*}
Clearly $\frac{1}{x} >0$. Consider 
\[ 
g(z) := 
 \dfrac{\omega \left( z , -r^{1/2} \right)}{ \sqrt{r} \omega \left( z , -r^{-1/2} \right)}
\]
for $z \in A_r$.
Then $g(z)$ is the conformal map from $A_r$ onto a circularly slit disk $\mathbb{D} \setminus L$, where $L$ is a proper subarc of a circle of radius $\sqrt{r}$ centred at $0$, with $g(-\sqrt{r})=0$ and $g(\partial \mathbb{D}) = \partial \mathbb{D}$ (see Section 5.6 of \cite{book_crowdy2020}). Furthermore, from the infinite product expression (\ref{eq:skpf_ar}) of $\omega ( z, \sqrt{r})$, one can deduce that 
$g(\overline{z}) = \overline{g(z)}$. In particular, it follows that $g(x) \in (-1,\sqrt{r})$ for any $x \in (-1,-r)$ and $g(x) \in (\sqrt{r},1)$ for any $x \in (r,1)$. Thus $g(x)-\sqrt{r}>0$ and $g(x)+\sqrt{r}>0$ for any $x \in (r,1)$. Therefore, $f(x) >0$ and hence by Remark \ref{rmk:fri_c_ar},
\[ H^c_{A_r} (z) > \dfrac{r}{|z|}\]
for any $z \in A_r$. Since $H^c_{A_r} (z) = H^c_{A_r} \left( \frac{r}{z} \right)$, we also have 
\[ H^c_{A_r} (z) > |z|  \]
for any $z \in A_r$.
It follows that
\[ H^c_{A_r} (z) > S_{A_r} (z)  \]
and hence 
\[ m^c_{A_r} (z) < 1  \]
for any $z \in A_r$. By inequality (\ref{eq:fri_order}), we also have 
\[ m^k_{A_r} (z) < 1  \]
for any $z \in A_r$. That 
\[ \lim\limits_{z \to p } m^d_{A_r} (z)  = 1 \]
for any $p \in \partial A_r$ is clear since 
\[ \lim\limits_{z \to p } H^d_{A_r} (z)  = 1 =\lim\limits_{z \to p } S_{A_r} (z) \]
for $d=k$ or $c$.
 
\end{proof}

Theorem \ref{thm:fri_k_disk} and \ref{thm:fri_c_disk} and together with the fact that $S_{\mathbb{D}^*} (z) =|z|$ allow us to obtain the following theorem which generalizes Theorem 6 in \cite{rong2020comparison}.

\begin{theorem}
\label{thm:quo_Disk}
For any $z \in \mathbb{D}^*$, we have
\[ m^k_{\mathbb{D}^*} (z)  < 1 
\qquad
\mbox{and}
\quad
m^c_{\mathbb{D}^*} (z)  = 1.
\]
When $z \to p$ for any $p \in \partial \mathbb{D}$, we have 
\[ \lim\limits_{z \to  p} m^d_{\mathbb{D}^*} (z)  = 1 \]
for $d=k$ or $c$ and when $z \to 0$, we have 
\[ \lim\limits_{z \to  0} m^c_{\mathbb{D}^*} (z)  = 1 
\qquad 
\mbox{and}
\qquad
\lim\limits_{z \to  0} m^k_{\mathbb{D}^*} (z)  = 0 
\]
\end{theorem}

\begin{proof}
By Theorem \ref{thm:fri_k_disk}, we have 
\[     H^k_{\mathbb{D}^*} (z) = 
    \dfrac{\log |z| + \sqrt{ (\log|z|)^2 + \pi^2} }{\pi}. \]
    Consider the function 
    \[ 
    f(t) = \pi e^{-t} - t + \sqrt{t^2 +\pi^2}  
\]
for $t \in [0,\infty)$. 
Since 
\[ f' (t) = - \pi e^{-t} - 1 + \dfrac{t}{\sqrt{t^2 +\pi^2}}, \]
we have $f'(t) <0$ for all $t \in (0,\infty)$ and hence $f(t)$ is decreasing for all $t \in (0,\infty)$. Since
\[ 
\lim\limits_{t \to \infty} f(t) 
=
\lim\limits_{t \to \infty} 
\left( 
\pi e^{-t} + \dfrac{\pi^2 }{t+\sqrt{t^2 +\pi^2}}
\right)
= 0,
 \]
we have $f(t) > 0$ for all $t \in [0,\infty).$ Putting $t=-\log |z|$, we have 
\[ H^k_{\mathbb{D}^*} (z) > H^c_{\mathbb{D}^*} (z) =S_{\mathbb{D}^*} (z) \]
and hence 
\[ m^k_{\mathbb{D}^*} (z)  < 1 
\qquad
\mbox{and}
\quad
m^c_{\mathbb{D}^*} (z)  = 1.
\]
Also, 
for any $p \in \partial \mathbb{D}$, 
we have \[ \lim\limits_{z \to  p} m^d_{\mathbb{D}^*} (z)  = 1 \]
for $d=k$ or $c$ because 
\[ 
\lim\limits_{z \to p} H^d_{\mathbb{D}^*} (z)  = 1 = \lim\limits_{z \to  p} S_{\mathbb{D}^*} (z) ,
\]
That $\lim\limits_{z \to  0} m^c_{\mathbb{D}^*} (z)  = 1$ is straightforward. 
Finally, using de L'h\^ospital rule, we have 
\begin{align*}
    \lim\limits_{z \to  0} m^k_{\mathbb{D}^*} (z)
    &=
    \lim\limits_{z \to  0} 
    \dfrac{\pi |z|}{\log |z| + \sqrt{ (\log |z| )^2 + \pi ^2 }}
    \\
    &=
    \lim\limits_{t \to  \infty} 
    \dfrac{\pi e^{-t}}{ -t + \sqrt{ t^2 + \pi ^2 }}
    \\
    &=
    \lim\limits_{t \to  \infty} 
    \dfrac{\pi e^{-t} (t + \sqrt{ t^2 + \pi ^2 }) }{ (t^2 + \pi^2) - t^2}
    \\
    &=
    \lim\limits_{t \to  \infty} 
    \dfrac{t + \sqrt{ t^2 + \pi ^2 } }{\pi e^{t}  }
    \\
    &=
    \lim\limits_{t \to  \infty} 
    \dfrac{ 1 + \frac{t}{\sqrt{ t^2 + \pi ^2 }}}{\pi e^{t}}
    \\
    &= 0.
\end{align*}
The result follows.
 
\end{proof}

\section{Boundary behavior of the Fridman function}
\label{sect:boundary_growth}

We will prove Theorem \ref{thm:fri:boundary_grow} in this section.

\begin{proof}
Since $X$ is a hyperbolic Riemann surface, by the Uniformization Theorem, $X$ is biholomorphic to $\mathbb{D} \diagup \Gamma$ where $\Gamma$ is a torsion-free discrete group of isometries of $\mathbb{D}$ (cf. Corollary 1.1.49 of \cite{aba}). Because the Fridman function is a biholomorphic invariant, we can assume without loss of generality that $X=\mathbb{D} \diagup \Gamma$ and hence Theorem \ref{thm:fri:main} applies. Let $\pi:D \to X$ be the quotient map. 

We first work on part 1 of the Theorem \ref{thm:fri:boundary_grow}. Since $X$ is of regular type and $\sigma$ has only one point $p$, then by Theorem 1.1.56  of \cite{aba}, there exists a point $q \in \partial \mathbb{D}$, which is a fixed point of some parabolic element $\widetilde{\gamma}$ in $\Gamma$, such that $\pi (q) =p$. Also by Theorem 1.1.56  of \cite{aba}, for any sequence $\{z_n\}$ in $X$ converges to $p$, there exists a sequence $\{ w_n \}$ in $\mathbb{D}$ with $\pi (w_n) = z_n$ for all $n$ such that the sequence $\{ w_n \}$ converges to $q$ non-tangentially (see for example p.428 of \cite{book:garnett} for definition of non-tangential limit). By Theorem \ref{thm:fri:main}, we have 
\begin{equation*}
    H^k_X (z) = \min\limits_{\gamma \in \Gamma \setminus \{ \mathrm{Id} \} } \tanh \frac{ \rho_\mathbb{D} (w,\gamma (w) )}{4}.
\end{equation*} Using the following hyperbolic trigonometry identities,
\[
\tanh \frac{x}{2} = \frac{\sinh x}{\cosh x +1}
\qquad
\mbox{and}
\qquad
\cosh^2 x  - \sinh^2 x =1,
\]
we have 
\begin{equation*}
    H^k_X (z) = \min\limits_{\gamma \in \Gamma \setminus \{ \mathrm{Id} \} }
   Q_\gamma (w)
\end{equation*} 
where \begin{equation*}
    Q_\gamma (w) = 
    \dfrac{\sinh \frac{1}{2}\rho_\mathbb{D} (w,\gamma (w) )}{1+\sqrt{1+ \sinh^2 \frac{1}{2}\rho_\mathbb{D} (w,\gamma (w) ) }} .
\end{equation*} 
Denote 
\begin{equation*}
    P(w,q)
    = 
\dfrac{1-|w|^2}{|w-q|^2}
\end{equation*}
the Poisson kernel on $\mathbb{D}$. If $\gamma \in \Gamma \setminus \{ \mathrm{Id} \}$ is parabolic, Theorem 7.35.1 in \cite{book:ddg} states that 
\begin{equation*}
\sinh \dfrac{1}{2} \rho_\mathbb{D} (w , \gamma(w))  
    = \dfrac{c_\gamma | w-t_\gamma|^2}{1-|w|^2}
\end{equation*}
where $c_\gamma$ is a constant depending on $\gamma$ and $t_\gamma \in \partial \mathbb{D}$ is the fixed point of $\gamma$.
In particular, for the parabolic element $\widetilde{\gamma} \in \Gamma$ which fixes $q$, we have 
\begin{equation*}
\sinh \dfrac{1}{2} \rho_\mathbb{D} (w , \widetilde{\gamma} (w))  
    = \dfrac{c_{\widetilde{\gamma}} | w-q |^2}{1-|w|^2}
\end{equation*}
When $w \to q $ non-tangentially,  $\frac{ |w- q|}{1-|w|}$ is bounded by definition (see p.428 of \cite{book:garnett}) and hence 
\[
\sinh \frac{1}{2} \rho_\mathbb{D} (w , \widetilde{\gamma} (w))   \leq k_{\widetilde{\gamma}} | w- q |
\]
for some constant $ k_{\widetilde{\gamma}} $. Therefore, when $w \to q $ non-tangentially,  $\sinh \frac{1}{2} \rho_\mathbb{D} (w , \widetilde{\gamma} (w)) \to 0$ and thus $Q_{\widetilde{\gamma}} (w) \to 0$. It follows that 
  \[ \lim\limits_{z \to p} H^k_X (z) =0. \]
Finally, as $S_X$ and $H^c_X$ are non-negative, by inequality (\ref{eq:fri_order}), we have 
  \[ \lim\limits_{z \to p} S_X (z)=\lim\limits_{z \to p} H^c_X (z) =0. \]
This proves part 1.
  
We now prove part 2. Since $X$ is of regular type and $\sigma$ has more than one point, then by Theorem 1.1.57  of \cite{aba}, there exists an open arc $\Sigma \subset \partial \mathbb{D}$ such that $\pi$ extends continuously to $\Sigma $ with $\pi (\Sigma )=\sigma $ and $\Gamma$ is properly discontinuous at every point of $\Sigma $. Then we can find a point $q \in \Sigma $ such that $\pi (q)=p$. For any sequence $\{z_n\}$ in $X$ convergent to $p$, we can find a sequence $\{w_n\}$ in $\mathbb{D}$ convergent to $q$ such that $\pi ( w_n) =z_n$ for all $n$. So
\begin{equation*}
   \lim\limits_{z \to p} H^k_X (z) = \lim\limits_{w \to q} \min\limits_{\gamma \in \Gamma \setminus \{ \mathrm{Id} \} } 
    Q_\gamma (w).
\end{equation*}
Let $\gamma$ be an element in $\Gamma \setminus \{ \mathrm{Id} \}$. Since $\Gamma$ is torsion-free, $\gamma$ is of infinite order. Also, $\gamma$ is not elliptic or otherwise $\Gamma$ is not discontinuous by Proposition 5.1.3 of \cite{book:hyper_geo}. Hence, 
$\gamma$ is either hyperbolic or parabolic.

If $\gamma$ is hyperbolic, then by Theorem 7.35.1 in \cite{book:ddg} we have
\begin{equation*}
     \sinh \dfrac{1}{2} \rho_{\mathbb{D}} (w , \gamma(w))
    = \cosh \rho_\mathbb{D} (w,A_\gamma) \sinh T_\gamma
\end{equation*}
where $A_\gamma$ is the axis of $\gamma$ and $T_\gamma$ is half of the translation length (which is a constant depending on $\gamma$). Let $a_\gamma \in A_\gamma  \subset \mathbb{D}$ be a point such that $\rho_{\mathbb{D}} (w , A_\gamma ) = \rho_{\mathbb{D}} ( w , a_\gamma )$. Notice  that $|a_\gamma| \neq 1$. Then we have
\begin{align*}
    \cosh \rho_\mathbb{D} (w ,a_\gamma) 
    & = 1 + \dfrac{2 |w-a_\gamma| ^2}{(1-|w|^2)(1-|a_\gamma|^2)}
    \\
    &=
    1 + \dfrac{k_{w,\gamma}}{1-|w|}
\end{align*}
where $ k_{w,\gamma} = \frac{2 |w-a_\gamma| ^2}{(1+|w|)(1-|a_\gamma|^2)} > c$ for some positive constant $c$ when $w$ is sufficiently close to $\partial \mathbb{D}$ as $a_\gamma \notin \partial \mathbb{D}$.
When $w \to q$, we have $\cosh \rho_\mathbb{D} (w ,a_\gamma) \to \infty$ and hence $     \sinh \frac{1}{2} \rho_{\mathbb{D}} (w , \gamma(w)) \to \infty$. Hence,  
\begin{align*}
\lim\limits_{w \to q} Q_\gamma (w) &=  \lim\limits_{w \to q} \dfrac{\sinh \frac{1}{2}\rho_\mathbb{D} (w,\gamma (w) )}{1+\sqrt{1+ \sinh^2 \frac{1}{2}\rho_\mathbb{D} (w,\gamma (w) ) }} \\
&=\lim\limits_{w \to q}
 \dfrac{1}{\left( \sinh \frac{1}{2}\rho_\mathbb{D} (w,\gamma (w) ) \right)^{-1}+\sqrt{\left( \sinh \frac{1}{2}\rho_\mathbb{D} (w,\gamma (w) ) \right)^{-2}+ 1 }} 
 \\
 &= 1.
\end{align*}
If $\gamma$ is parabolic, Theorem 7.35.1 in \cite{book:ddg} states that 
\begin{equation*}
\sinh \dfrac{1}{2} \rho_\mathbb{D} (w , \gamma(w))  
    = \dfrac{c_\gamma | w-t_\gamma|^2}{1-|w|^2}
\end{equation*}
where $c_\gamma$ is a constant depending on $\gamma$ and $t_\gamma \in \partial \mathbb{D}$ is the fixed point of $\gamma$. Since $l$ contains more than one point, we have $q$ is not a parabolic fixed point, i.e, $t_\gamma \neq q$ (see for example Proposition 1.1.58 of Abate \cite{aba}). Then $\sinh \dfrac{1}{2} \rho (w , \gamma(w)) \to \infty$ as $w \to q$. Hence,  
\begin{align*}
\lim\limits_{w \to q} Q_\gamma (w) &=  \lim\limits_{w \to q} \dfrac{\sinh \frac{1}{2}\rho_\mathbb{D} (w,\gamma (w) )}{1+\sqrt{1+ \sinh^2 \frac{1}{2}\rho_\mathbb{D} (w,\gamma (w) ) }} \\
&=\lim\limits_{w \to q}
 \dfrac{1}{\left( \sinh \frac{1}{2}\rho_\mathbb{D} (w,\gamma (w) ) \right)^{-1}+\sqrt{\left( \sinh \frac{1}{2}\rho_\mathbb{D} (w,\gamma (w) ) \right)^{-2}+ 1 }} 
 \\
 &= 1.
\end{align*}
In any cases, $\lim\limits_{w \to q} Q_\gamma (w) =1$. Therefore,
\[ \lim\limits_{z \to p} H^k_X (z) =1. \]
This proves part 2.
 
\end{proof}

\section{Appendix : alternative definition for the squeezing function }
\label{appendix:alt_def_squeezing}

Let $X \subset \mathbb{C}^n$ be a bounded domain. In 2012, Deng, Guan and Zhang \cite{squ_def1} defined the squeezing function $S_X (z)$ to be 
\begin{equation*}
    S_X (z) 
:= 
\sup
\left\lbrace 
\frac{a}{b} 
\: : \:
\mathbb{B}^n (0;a) \subset  f(\Omega ) \subset \mathbb{B}^n(0;b),
f\in \mathcal{U}(X,\mathbb{C}^n), f(z)=0
\right\rbrace
\end{equation*}
for each $z \in X$. The following lemma gives a reformulation of $S_X (z)$.

\begin{lemma}
\label{lem:sq:reformulate}
Define
\begin{equation*}
    \widehat{S}_X (z) 
= 
\sup
\left\lbrace 
\tanh \frac{r}{2}
\: : \:
B^k_{\mathbb{B}^n} (f(z);r) \subset  f(X )
, f \in \mathcal{U} (X, \mathbb{B}^n )
\right\rbrace.
\end{equation*}
Then $S_X (z) = \widehat{S}_X (z)  $.
\end{lemma}

\begin{proof}
Fix any $f \in \mathcal{U}(X,\mathbb{C}^n)$ such that $f(z)=0$. There exists constant $a,b$ such that $\mathbb{B}^n (0;a) \subset  f(\Omega ) \subset \mathbb{B}^n(0;b)$. Define $g(z):=\frac{f(z)}{b}$. Notice that $g \in \mathcal{U}(X,\mathbb{B}^n)$ and $g(z)=0$. 
It follows that 
\begin{align*}
    S_X (z) 
&=
\sup
\left\lbrace 
\frac{a}{b} 
\: : \:
\mathbb{B}^n \left( 0; \frac{a}{b} \right) \subset  f(X ) \subset \mathbb{B}^n,
f \in \mathcal{U}(X,\mathbb{C}^n), f(z)=0
\right\rbrace
\\
&=
\sup
\left\lbrace 
a
\: : \:
\mathbb{B}^n (0; a) \subset  g(X) ,
g \in \mathcal{U}(X,\mathbb{B}^n), g(z)=0
\right\rbrace.
\end{align*}
Note that for any $r$ such that $0<r<1$, we have $ \mathbb{B}^n \left(0 ; \tanh \frac{r}{2} \right)
    =
    B^k_{\mathbb{B}^n} (0; r)$.
It follows that
\begin{equation*}
  S_X (z)
    =
\sup
\left\lbrace 
\tanh \frac{r}{2}
\: : \:
B^k_{\mathbb{B}^n} (0; r) \subset  g(X ) ,
g \in \mathcal{U}(X,\mathbb{B}^n), g(z)=0
\right\rbrace.
\end{equation*}

Now for any $F \in \mathcal{U}(X, \mathbb{B}^n$), there exists an automorphism $\phi:\mathbb{B}^n \to \mathbb{B}^n$ of $\mathbb{B}^n$ such that 
$\phi(F(z))=0$ (see, for instance, Theorem 2.2.2. of \cite{book:rudin_ball}).
Applying the contraction property of Kobayashi metric to $\phi$ as well as $\phi^{-1}$, we have $\phi$ is an isometry. It follows that
\begin{equation*}
    \phi \left( B^k_{\mathbb{B}^n} (F(z);r) \right) = B^k_{\mathbb{B}^n} ( 0;r).
\end{equation*}
Defining $g=\phi \circ F$, we have 
$S_X (z) = \widehat{S}_X (z)  $.
 
\end{proof}

\begin{remark}
One can replace $B^k_{\mathbb{B}^n} (F(z);r)$ by $B^c_{\mathbb{B}^n} (F(z);r)$, the reformulation still works. 
\end{remark}

\bigskip
\paragraph{Acknowledgments:}
The first author was partially supported by the RGC grant 17306019. The second author was partially supported by a HKU studentship and the RGC grant 17306019. 

\bibliographystyle{spmpsci}   
\bibliography{reference}
\label{pg:reference}
\end{document}